\documentclass[12pt]{amsart}
\usepackage[margin=1.0in]{geometry}
\usepackage{amsmath,amssymb,amsfonts,graphicx,color,fancyhdr,psfrag,comment,enumerate,mathabx}
\usepackage[latin1]{inputenc}
\usepackage[hyperpageref]{backref}
\usepackage[colorlinks=true, pdfstartview=FitV, linkcolor=blue, citecolor=blue, urlcolor=blue]{hyperref}
\usepackage[capitalise,noabbrev]{cleveref}
\usepackage{tcolorbox,longfbox}
\allowdisplaybreaks
\usepackage{epstopdf}
\usepackage{mathrsfs}
\usepackage{epsfig}
\usepackage{hyperref}

\usepackage{ifthen}
\usepackage{float}
\usepackage{enumerate}
\usepackage{mathtools}
\usepackage[bottom]{footmisc}
\usepackage{tikz}
\usepackage{comment}
\usetikzlibrary{matrix,shapes,arrows,positioning,chains}
\usepackage{amssymb,amsmath,amsfonts}
\usepackage{bbm}

\numberwithin{equation}{section}

\newtheorem{theorem}{Theorem}[section]
\newtheorem{proposition}[theorem]{Proposition}
\newtheorem{lemma}[theorem]{Lemma}

\newtheorem{remark}[theorem]{Remark}
\newtheorem{corollary}[theorem]{Corollary}

\newcommand{\R}{\mathbb{R}}
\newcommand{\C}{\mathbb{C}}

\begin{document}
\title[On Sobolev spaces for Grushin operator]{On Sobolev spaces for the Grushin operator}

\author[N. Garg and R. Garg]
{Nishta Garg \and Rahul Garg}

\address[N. Garg]{Department of Mathematics, Indian Institute of Science Education and Research Bhopal, Bhopal--462066, Madhya Pradesh, India.}
\email{nishta21@iiserb.ac.in}

\address[R. Garg]{Department of Mathematics, Indian Institute of Science Education and Research Bhopal, Bhopal--462066, Madhya Pradesh, India.}
\email{rahulgarg@iiserb.ac.in}


\begin{abstract}
We define Sobolev spaces associated with the Grushin operator with or without drift. We study some of their norm equivalences and establish embedding and algebra properties. We also prove the boundedness of the Riesz transforms on these spaces. 
\end{abstract}

\maketitle

\section{Introduction}
Sobolev spaces are fundamental in the study of partial differential equations, functional analysis, and variational calculus. They provide a quantitative way to measure the smoothness of a given function in terms of the integrability of its (weak) derivatives. Owing to their vast applications, Sobolev spaces are extensively studied in various settings, including Euclidean spaces, Riemannian manifolds, Lie groups, and doubling metric measure spaces. For our purpose, let us refer to just a few selected ones \cite{Sobolev_embedding_Lie_groups, Sobolev-Embedding_Constant_Lie_group_BPV22, Multipliers_on_fractional_Sobolev_Space_Strichartz, Sobolev_algebras_on_nonunimodular_Lie_groups_Vallarino,  Algebra_properties_Bohnke, Sobolev_algebras_on_Lie_groups_and_Riemannian_Manifold_Coulhon, Bruno-Homogeneous-algebras-via-heat-kernel-estimates-2022, bui2025bilinearfractionalleibnizrules,Bui_Huy_Duong_weightes_Besov_and_Triebel_Lizorkin_spaces_2020}.

\medskip 
In the present article, we define and study non-homogeneous Sobolev spaces associated with the Grushin operator with or without drift. Let us remark that stronger results for homogeneous function spaces and that too in a much more general setup are recently established in \cite{Bruno-Homogeneous-algebras-via-heat-kernel-estimates-2022, bui2025bilinearfractionalleibnizrules,Bui_Huy_Duong_weightes_Besov_and_Triebel_Lizorkin_spaces_2020}, but those results are only in the setup of no non-trivial drift. Our work here deals with a unified study of non-homogeneous Sobolev spaces with or without a drift vector. In that, our motivation comes from the works done in \cite{Sobolev_algebras_on_Lie_groups_and_Riemannian_Manifold_Coulhon, Sobolev_embedding_Lie_groups, Sobolev_algebras_on_nonunimodular_Lie_groups_Vallarino}. 

\medskip 
Before moving to our context, let us first discuss the relevant literature on some Lie groups. Let $G$ be a non-compact connected Lie group with identity $e$ and $X = \{X_{1}, \ldots, X_{l}\}$ be linearly independent left-invariant vector fields on $G$ which satisfy H\"{o}rmander's condition. Let $\sigma$ be the right Haar measure on $G$, $\delta$ the modular function, and take $\chi$ to be a continuous positive character on $G$. Consider the measure $\mu_{\chi}$ on $G$ whose density is $\chi$ with respect to $\sigma$. Let $c_j = (X_j \chi)(e)$, for $j = 1, \ldots, l$, and consider the left-invariant sub-Laplacian with drift $\Delta_{\chi}$ on $G$ which is given by 
\begin{align}
\label{def:sub-Lap-with-drift}
\Delta_{\chi} = -\sum_{j=1}^{l} (X_{j}^{2} + c_j X_j).
\end{align}
When $\chi$ is the trivial character, $\Delta_{\chi}$ is the usual sub-Laplacian $\Delta_{\chi} = \Delta = -\sum_{j=1}^{l} X_{j}^{2}$. 

\medskip 
The authors in \cite{Hebisch-Mauceri-Meda-spectral-multipliers-drift-Lie-group-Math-Z-2005}  studied spectral multipliers for $\Delta_{\chi}$, where they first showed that $\Delta_{\chi}$ is essentially self-adjoint on $L^{2}(d\mu_{\chi})$. Recently, Bruno et al. \cite{Sobolev_embedding_Lie_groups} introduced Sobolev spaces $L^{p}_{\alpha}(d\mu_{\chi})$ and studied various of their properties. For $1 < p < \infty$ and $\alpha \geq 0$, consider the space
$$
L^{p}_{\alpha}(d\mu_{\chi}) := \{ f \in L^{p}(d\mu_{\chi}): \, \Delta_{\chi}^{\alpha/2} f \in L^{p}(d\mu_{\chi})\},
$$
endowed with the norm
$$
\|f\|_{L^{p}_{\alpha}(d\mu_{\chi})} := \|f\|_{L^{p}(d\mu_{\chi})} + \|\Delta_{\chi}^{\alpha/2} f\|_{L^{p}(d\mu_{\chi})}, 
$$
with the usual understanding that for $\alpha = 0$, the space $L^{p}_{\alpha}(d\mu_{\chi})$ is nothing but $L^{p}(d\mu_{\chi})$. 

\medskip 
They established the following embedding and algebra properties of these spaces. Let $d_C$ denote the Carnot-Carath\'eodory distance on $G$ and $d_0$ stands for the local dimension of the metric space $(G, d_C, \sigma)$. 

\begin{theorem}[\cite{Sobolev_embedding_Lie_groups}] \label{thm:embeddings_group_setup}
Let $\chi$ be a positive character of $G$ and $1 < p, \, q < \infty$.
\begin{enumerate}
\item If $\alpha > 0$, then $L^{p}_{\alpha}(d\mu_{\chi}) \hookrightarrow L^{q}(d\mu_{\chi^{q/p} \delta^{1-q/p}})$ for every $q \geq p$ such that $\frac{1}{p}-\frac{1}{q} \leq \frac{\alpha}{d_0}$.
\item If $\alpha \geq d_0/p$, then $L^{p}_{\alpha}(d\mu_{\chi}) \hookrightarrow L^{q}(d\mu_{\chi^{q/p} \delta^{1-q/p}})$ for every $q \geq p$.
\item If $\alpha > d_0/p$, then $L^{p}_{\alpha}(d\mu_{\chi}) \hookrightarrow (\delta \, \chi^{-1})^{1/p} \, L^{\infty}$.
\end{enumerate}
\end{theorem}

\begin{theorem}[\cite{Sobolev_embedding_Lie_groups}] \label{thm:algebra_properties_group_setup}
Let $\chi$ be a positive character of $G$ and $\alpha \geq 0$. Let $p_1, \, q_2 \in (1,\infty]$ and $p, \, p_2, \, q_1 \in (1,\infty)$ be such that $\frac{1}{p} = \frac{1}{p_1} + \frac{1}{q_1} = \frac{1}{p_2} + \frac{1}{q_2}$. Then, 
$$
\|fg\|_{L^{p}_{\alpha}(d\mu_{\chi})} \lesssim \|f\|_{L^{p_1}(d\mu_{\chi})} \|g\|_{L^{q_1}_{\alpha}(d\mu_{\chi})} + \|f\|_{L^{p_2}_{\alpha}(d\mu_{\chi})} \|g\|_{L^{q_2}(d\mu_{\chi})}, 
$$
for all $f \in L^{p_1}(d\mu_{\chi}) \cap L^{p_2}_{\alpha}(d\mu_{\chi})$ and $g \in L^{q_1}_{\alpha}(d\mu_{\chi}) \cap L^{q_2}(d\mu_{\chi})$. In particular, $L^{p}_{\alpha}(d\mu_{\chi}) \cap L^{\infty}(d\mu_{\chi})$ is an algebra for every $p \in (1,\infty)$.
\end{theorem}

Recall now that writing the points on the space $\R^{d_1+d_2} (= \R^{d_1} \times \R^{d_2})$ as $x=(x',x'') = (x_1', \ldots, x_{d_1}', x_1'', \ldots, x_{d_2}'')$, one considers the Grushin operator $G = - \Delta_{x'} - |x'|^2 \Delta_{x''}$, which is known to be a hypoelliptic operator and is homogeneous of degree 2 with respect to the non-isotropic dilations $\delta_{r} x = (r x', r^2 x'')$. We denote by $Q$ the corresponding homogeneous dimension $Q = d_1+2d_2$. Various aspects of this operator are extensively studied in the literature. In particular, works concerning its spectral multipliers, Riesz transforms and pseudo-multipliers can be found in \cite{Riesz_transform_Robinson, Riesz-trans-multipliers-Grushin-oper-Jotsaroop-Sanjay-thangavelu-JDM-2014, Martini-Sikora-sharp-multiplier-Grushin-MRL-2012, Martini-Muller-sharp-multiplier-Grushin-Revista-2014, Dallara-Martini-robust-approach-multiplier-grushin-TAMS-2020, Dallara-Martini-aptimal-multiplier-grushin-part-II-JFAA-2022, Dallara-Martini-aptimal-multiplier-grushin-part-I-Revista-2023, Bagchi-Garg-L2-boundedness-pseudo-multiplier-Grushin-JFA-2024, Bagchi-Basak-Garg-Ghosh-sparse-pseudo-multiplier-grushin-I-JFAA-2023, Bagchi-Basak-Garg-Ghosh-sparse-pseudo-multiplier-grushin-II-JGA-2024, Bagchi-Basak-Garg-Ghosh-sparse-pseudo-multiplier-grushin-III-JMAA-2024, Dim-free-Riesz-trans-Grushin-oper-PAMS-Sanjay-thangavelu-2014, Dziubanski-Jotsaroop-H1-BMO-Grushin-JFAA-2016, Dziubanski-Sikora-Lie-group-approach-JLT-2021}. 

\medskip In terms of the vector field $X = (X', X'') = (X_1, \ldots, X_{d_1}, \, X_{1,1}, \ldots, X_{d_1,d_2})$, where 
$$
X_{j} = \frac{\partial}{\partial x'_{j}} \quad \text{and} \quad X_{j,k} = x'_{j} \frac{\partial}{\partial x''_{k}}, 
$$
the operator $G$ can be expressed as the negative of the sum of their squares: 
\begin{equation}
G = - X \cdot X = - X' \cdot X' - X'' \cdot X'' = - \sum_{j=1}^{d_1} X_{j}^{2} - \sum_{j=1}^{d_1} \, \sum_{k=1}^{d_2} X_{j,k}^{2}. 
\end{equation}

In a recent work \cite{Garg-Garg-Riesz-transform-Grushin-drift-2026}, we introduced the Grushin operator with drift  $G_{(\nu, b)}$, which is defined as follows. For a non-zero vector $(\nu, b) = (\nu_1, \ldots, \nu_{d_1}, \, b_{1,1}, \ldots, b_{d_1,d_2}) \in \R^{d_1} \times \R^{d_1d_2}$, let us consider the following operator: 
\begin{align}
G_{(\nu,b)} 
&= G - 2 \, (\nu, b) \cdot X \\ 
\nonumber &= -\sum_{j=1}^{d_1} \frac{\partial^2}{\partial x_j'^2} - \sum_{j=1}^{d_1} \sum_{k=1}^{d_2} x_j'^2 \frac{\partial^2}{\partial x_k''^2} - 2 \, \sum_{j=1}^{d_1} \nu_{j} \, \frac{\partial}{\partial x_j'} - 2 \, \sum_{j=1}^{d_1} \sum_{k=1}^{d_2} b_{j,k} \, x_{j}' \frac{\partial}{\partial x_k''}. 
\end{align}

\medskip 
It was proved in \cite[Proposition 2.3]{Garg-Garg-Riesz-transform-Grushin-drift-2026} that $G_{(\nu, b)}$ is symmetric with respect to a positive measure $\mu$ on $\R^{d_1+d_2}$ if and only if $b = 0$, and in this case, upto a scalar multiple, the measure $\mu$ is given by $d\mu = d\mu_{\nu} = e^{2\nu \cdot x'} \, dx$, where $dx$ denote the Lebesgue measure on $\R^{d_1+d_2}$.
Henceforth, we shall denote by $\nu = (\nu_{1}, \ldots, \nu_{d_1})$ a vector in $\R^{d_1}$ and the Grushin operator with drift as 
$$ G_\nu = G_{(\nu, 0)} = G - 2 \, \nu \cdot X' = G - 2 \, \nu \cdot \nabla_{x'},$$ 
duly noticing that for $\nu = 0$, the operator $G_\nu$ reduces to the Grushin operator $G$ with the measure $d\mu_0 (x) = dx$. 

\medskip 
It also turns out that the operator $G_\nu$ is positive-definite and essentially self-adjoint  on $L^2(\R^{d_1+d_2}, d\mu_\nu)$. With abuse of notation, we still denote by $G_\nu$ the extension of $G_\nu$ on $L^p(\R^{d_1+d_2}, d\mu_\nu)$. It was also shown in \cite{Garg-Garg-Riesz-transform-Grushin-drift-2026} that when $\nu \neq 0$, the measure $d\mu_{\nu}$ is of exponential volume growth. Let $B(x, r)$ denote the open ball with center $x$ and radius $r>0$ with respect to the control distance $\rho$ of $G$. We denote the ball volume of $B(x, \sqrt{r})$ with respect to the measure $\mu_\nu$ by $V_\nu(x,r)$. 

\medskip For $1<p<\infty$ and $\alpha \geq 0$, we define the Sobolev space
\begin{align} 
\label{def:sobolev-drift-main}
L^p_\alpha(d\mu_{\nu}) := \{f \in L^p(d\mu_\nu): G_{\nu}^{\alpha/2}f \in L^p(d\mu_\nu)\},
\end{align}
endowed with the norm 
$$
\|f\|_{L^p_\alpha(d\mu_{\nu})} := \|f\|_{L^p(d\mu_\nu)}+\|G_{\nu}^{\alpha/2}f\|_{L^p(d\mu_\nu)},
$$
with the usual understanding that for $\alpha = 0$, the space $L^p_\alpha(d\mu_{\nu})$ is nothing but $L^p(d\mu_{\nu})$. 

\medskip Here is our main result concerning their embedding properties. 

\begin{theorem} \label{thm:embedding}
Let $\nu \in \R^{d_1}$ and $V_{\nu}$ be as defined above (unless stated otherwise, $\nu$ can also be zero). 
\begin{enumerate}[(i)]
\item (Sufficient condition) Let $1<p \leq q<\infty$ and $\frac{1}{p}-\frac{1}{q} \leq \frac{\alpha}{Q}$. Then, 
$$
\|V_\nu(\cdot,1)^{\frac{1}{p}-\frac{1}{q}} \, f\|_{L^q(d\mu_\nu)} \lesssim \|f\|_{L^p_\alpha(d\mu_\nu)}.
$$

\item (Sufficient condition) Let $1<p<\infty$ and $\alpha>Q/p.$ Then, 
$$
\|V_\nu(\cdot,1)^{\frac{1}{p}} f \|_{L^\infty(d\mu_\nu)} \lesssim \|f\|_{L^p_\alpha(d\mu_\nu)}.
$$

\item (Necessary condition) 
Let $\nu \neq 0$. Given $1 < p < \infty, \, 1 \leq q \leq \infty$ and $\eta \geq 0$, for the embedding 
$$
\|V_\nu(\cdot,1)^{\eta} \, f\|_{L^q(d\mu_\nu)} \lesssim \|f\|_{L^p_\alpha(d\mu_\nu)}, 
$$
to hold true, we must have $\eta = \frac{1}{p} - \frac{1}{q}$. \end{enumerate}
\end{theorem}

\begin{remark}
\label{rem:appearance-ball-volume}
We have the following remarks on Theorem \ref{thm:embedding}. 
\begin{enumerate}
\item We could prove the necessary part (iii) in Theorem \ref{thm:embedding} only for $\nu \neq 0$. We strongly believe that in the case when $\nu = 0$, a similar necessary condition $\eta \leq \frac{1}{p} - \frac{1}{q}$ should hold true, but right now we do not know how to prove it. 

\medskip 
\item 
Note that in parts (i) and (ii) of Theorem \ref{thm:embedding}, we have terms of the form $V_{\nu}(\cdot, 1)^{\eta}$ in the left-hand side of the inequalities, and part (iii) establishes the explicit importance of such a term. In the context of Lie groups \cite[Theorem 1.1]{Sobolev_embedding_Lie_groups}, this type of embedding works with the introduction of appropriate powers of $\chi(x)$. In this setting, the ball volume factor $\mu_{\chi}(B(x, 1))$ is nothing but equivalent to $\chi(x)$. So, it gives a similar result. Moreover, part (iii) of the above theorem can be compared to the corresponding result in \cite{Sobolev_embedding_Lie_groups}, which asserts that for $p \neq q$, the embedding of the form $L^{p}_{\alpha}(d\mu_{\chi}) \hookrightarrow L^{q}(d\mu_{\chi})$ can not hold unless $\mu_{\chi}$ is the left Haar measure. 

\medskip 
While working with Grushin operator $G$ (without drift), in \cite[Theorem 2.6]{Bagchi-Basak-Garg-Ghosh-sparse-pseudo-multiplier-grushin-II-JGA-2024}, the authors also proved that a ball volume factor can be included in the Sobolev embedding. Note that since $|B(\cdot, 1)| \gtrsim 1$, the presence of this extra ball volume factor $|B(\cdot, 1)|^{\frac{1}{p}-\frac{1}{q}}$ leads to a stronger result than the usual embedding $L^{p}_{\alpha}(dx) \hookrightarrow L^{q}(dx)$. 
\end{enumerate}
\end{remark}

Analogous to Theorem \ref{thm:algebra_properties_group_setup}, we have the following algebra property of $L^p_{\alpha}(d\mu_\nu)$-spaces. 

\begin{theorem} \label{thm:algebra_properties}
Let $\alpha \geq 0,  \, p_1, q_2 \in (1, \infty]$ and $p, p_2, q_1 \in (1, \infty)$ be such that $\frac{1}{p} = \frac{1}{p_1}+\frac{1}{q_1} = \frac{1}{p_2}+\frac{1}{q_2}$. Then, 
\begin{align*}
\|fg\|_{L^p_\alpha(d\mu_\nu)} \lesssim \|f\|_{L^{p_1}(d\mu_\nu)} \|g\|_{L^{q_1}_{\alpha}(d\mu_\nu)} + \|f\|_{L^{p_2}_{\alpha}(d\mu_\nu)} \|g\|_{ L^{q_2}(d\mu_\nu)},    
\end{align*} 
for all $f \in L^{p_1}(d\mu_\nu) \cap L^{p_2}_{\alpha}(d\mu_\nu)$ and $g \in L^{q_1}_{\alpha}(d\mu_\nu) \cap L^{q_2}(d\mu_\nu)$. Consequently, $L^p_\alpha(d\mu_\nu) \cap L^\infty(d\mu_\nu)$ is an algebra for every $p \in (1, \infty).$
\end{theorem}

It is not difficult to prove Theorem \ref{thm:algebra_properties} following the ideas of the proof of \cite[Theorem 1.2]{Sobolev_embedding_Lie_groups}, so we shall not write its proof. 

\medskip
Before delving into the proof of the embedding theorem (Theorem \ref{thm:embedding}) in Section \ref{sec:proof-embedding}, we study some important properties of these Sobolev spaces, such as various of their norm equivalences. We take it up in Section \ref{sec:properties-sobolev-spaces} and the same is done with the help of the boundedness properties of some local Riesz transforms associated with $G_\nu$. As a byproduct, we shall also prove the Sobolev space boundedness of the Riesz transforms of arbitrary order associated with $G_\nu$. Given a multi-index $0 \neq \gamma \in \left(\mathbb{N} \cup \{0\} \right)^{d_1+d_1d_2}$, consider the Riesz transform $R_\gamma$ defined by $R_\gamma = X^{\gamma} \, G_{\nu}^{-|\gamma|/2}$. For $1 < p < \infty$, the $L^{p}(d\mu_{\nu})$-boundedness of these Riesz transforms is well known in the literature in the case when $\nu = 0$, and the same was shown to hold true for any $\nu \neq 0$ in \cite{Garg-Garg-Riesz-transform-Grushin-drift-2026}. We shall prove the following result confirming their boundedness on Sobolev spaces $L^{p}_{\alpha}(d\mu_{\nu})$ as well.
\begin{theorem} 
\label{thm:Riesz-transform-boundedness-Sobolev}
Let $1 < p < \infty$ and $\alpha \geq 0$. For any multi-index $\gamma$, the Riesz transform $R_{\gamma} = X^{\gamma} \, G_{\nu}^{-|\gamma|/2}$ is bounded on $L^{p}_{\alpha}(d\mu_{\nu})$.    
\end{theorem}

In the next section, we shall begin with recalling some basic information concerning the operator $G_\nu$, including properties of the heat kernel as well as Riesz transforms. As mentioned earlier, the main results will be proved in Sections \ref{sec:properties-sobolev-spaces} and \ref{sec:proof-embedding}.  


\section{Preliminaries} 
\label{sec:prelim}

As mentioned in the introduction, given a vector $\nu \in \R^{d_1}$, the Grushin operator with drift $G_{\nu}$ is defined on $\R^{d_1+d_2}$ by 
\begin{equation} \label{def:Grushin_with_drift}
G_\nu = -\sum_{j=1}^{d_1} X_j^2 - \sum_{j=1}^{d_1} \sum_{k=1}^{d_2} X_{j,k}^2 - 2 \, \sum_{j=1}^{d_1} \nu_j X_j, 
\end{equation}
where 
$X_{j} = \frac{\partial}{\partial x'_{j}}$ and $X_{j,k} = x'_{j} \frac{\partial}{\partial x''_{k}},$ for $1 \leq j \leq d_1$ and $1 \leq k \leq d_2$. 

\medskip 
The operator $G_\nu$ is positive-definite and essentially self-adjoint on $L^2(\R^{d_1+d_2}, d\mu_\nu)$, where $d\mu_\nu = e^{2\nu\cdot x'} \, dx$, and for $\nu = 0$, the operator $G_\nu$ reduces to the standard Grushin operator $G_0 = G$ with the measure $d\mu_0 = dx$, the Lebesgue measure on $\R^{d_1+d_2}$. 

\medskip 
The control distance $\rho(x,y)$ for the Grushin operator $G$ is known to have the following asymptotic behavior (see \cite{Analysis_of_degenerate_elliptic_opertators_Robinson}): 
\begin{equation} \label{Grushin_distance}
\rho(x,y) \sim  |x'-y'| + \left\{
\begin{array}{ll}
\frac{|x''-y''|}{|x'|+|y'|}, & \mbox{if } |x''-y''|^{\frac{1}{2}} \leq |x'|+|y'| \\
|x''-y''|^{\frac{1}{2}},& \mbox{if } |x''-y''|^{\frac{1}{2}} \geq |x'|+|y'|.
\end{array} \right.    
\end{equation} 
For convenience, we shall refer the right-hand side of \eqref{Grushin_distance} as the Grushin metric and will denote it by $\rho$ itself.

\medskip 
If $B(x,r) = \, \{y \in \R^{d_1+d_2}: \rho(x,y)<r\}$ denotes the open ball with center $x$ and radius $r>0$, then it is well known from \cite{Analysis_of_degenerate_elliptic_opertators_Robinson} that 
\begin{equation} \label{eq:ball_volume}
|B(x,r)| \sim  r^{d_1+d_2} \, \text{max}\{r,|x'|\}^{d_2} \sim r^{d_1+d_2}(r+|x'|)^{d_2}, 
\end{equation}
where $|\cdot|$ denotes the Lebesgue measure on $\R^{d_1+d_2}$. It follows from \eqref{eq:ball_volume} that the Lebesgue measure on the metric space $(\R^{d_1+d_2}, \, \rho)$ satisfies the doubling property with the doubling constant $Q = d_1+ 2d_2$. 

\medskip 
On the other hand, when $\nu \neq 0$, it was shown in \cite[Lemma 2.4]{Garg-Garg-Riesz-transform-Grushin-drift-2026} that the measure $d\mu_{\nu}$ is of exponential volume growth. More precisely, if we denote the ball volume of the open ball $B(x, \sqrt{r})$ in this measure by $V_\nu(x,r)$, then $V_\nu$ satisfies the following asymptotics: 
\begin{align}
\label{main:ball-vol-est}
V_\nu(x,r) \sim 
\left\{
\begin{array}{ll}
e^{2\nu \cdot x'} \, r^{(d_1+d_2)/2} \, (\sqrt{r}+|x'|)^{d_2}, & \mbox{if } \sqrt{r} \leq 1/|\nu| \\
|\nu|^{-\frac{(d_1+1)}{2} - d_2} \, e^{2(\nu \cdot x'+|\nu|\sqrt{r})} \, r^\frac{d_1-1}{4} \, (\sqrt{r}+|x'|)^{d_2}, & \mbox{if } \sqrt{r} > 1/|\nu|.
\end{array} \right.  
\end{align}

\medskip 
For any $\nu \in \R^{d_1}$, the operator $G_\nu$ generates a symmetric diffusion heat semigroup $(e^{-t G_\nu})_{t>0}$ on $L^2(d\mu_\nu)$, which is given by 
$$
e^{-tG_{\nu}} f(x) = \int_{\R^{d_1+d_2}} H_{t, \nu}(x,y) \, f(y) \, d\mu_{\nu}(y),
$$
and the heat kernel $H_{t, \nu}$ for a general $\nu$ is given in terms of the heat kernel $H_t = H_{t, 0}$ for $\nu = 0$ in the following explicit manner: 
$$
H_{t, \nu}(x,y) = e^{-|\nu|^2 t} \, e^{-\nu \cdot (x'+ y')} \, H_{t}(x,y). 
$$ 
It is well known that the heat kernel $H_{t}$ of the Grushin operator $G$ satisfies the Gaussian upper and lower bounds, that is, there exist constants $b', \, b'' > 0$ such that 
$$
|B(x,\sqrt{t})|^{-1} \, e^{-b'' \rho(x,y)^2 / t} \lesssim H_{t}(x,y) \, \lesssim |B(x,\sqrt{t})|^{-1} \, e^{-b' \rho(x,y)^2 / t} .
$$
By Hille-Yosida theorem (see \cite[II.3.5]{One-parameter_semigroups_for_linear_evolution_equations_2000}), the operator $G_\nu$ is densely defined, closed, its resolvent contains $(-\infty,0)$ and for all $t>0$, 
\begin{align*} 
\|t \, (tI + G_\nu)^{-1}\|_{L^p(d\mu_\nu) \to L^p(d\mu_\nu)} \leq 1.
\end{align*} 
Thus, we can define fractional powers $G_\nu^\alpha$, for any $\alpha \in \C$. These operators are densely defined and satisfy $G_\nu^{\alpha + \beta} = G_\nu^{\alpha} \, G_\nu^{\beta}$ for all $\alpha, \beta \in \C$ (see, for example, \cite{Functional_analysis_and_semi_groups_Hille, Fractional_powers_of_operators}).

\begin{remark} \label{rem:1}
Let $1 \leq p \leq \infty$ and $\alpha > 0$. Given any $c > 0$, it follows from the contraction property of the heat semigroup $(e^{-t G_\nu})_{t>0}$ and the functional calculus identity
\begin{align} 
\label{integral-identity-negative-fractional-power}
(G_{\nu}+cI)^{-\alpha} = \frac{1}{\Gamma(\alpha)} \int_{0}^{\infty} t^{\alpha-1} \, e^{-c \, t} \, e^{-tG_{\nu}} \, dt, 
\end{align}
that the operator $c^\alpha (G_{\nu}+cI)^{-\alpha}$ is bounded on $L^{p}(d\mu_{\nu})$ uniformly in $c > 0$.

\medskip 
It immediately follows from the above that the operator $G_\nu^{\alpha} (G_\nu + cI)^{-\alpha}$ is bounded on $L^{p}(d\mu_{\nu})$ uniformly in $c>0$ for any $\alpha \in \mathbb{N}$. In fact, $G_\nu^{\alpha} (G_\nu + cI)^{-\alpha}$ is bounded on $L^{p}(d\mu_{\nu})$ uniformly in $c>0$ for every $\alpha > 0$. Clearly, we only need to verify it for $\alpha \in (0,1)$, and the same can be seen, for example, via the following identity: 
$$
G_\nu^{\alpha} (G_\nu + cI)^{-\alpha} = C_{\alpha}^{-1} \int_{0}^{\infty} s^{-1+\alpha} \, (1+s)^{-1} \, G_{\nu} \left( G_{\nu} + s (1+s)^{-1} c I \right)^{-1}  \, ds,
$$
where $C_{\alpha} = \int_{0}^{\infty} s^{-1+\alpha} \, (1+s)^{-1} \, ds$.
\end{remark}

\begin{remark} \label{rem:2}
For any multi-index $0 \neq \gamma \in \left(\mathbb{N} \cup \{0\} \right)^{d_1+d_1d_2}$, let $R_\gamma = X^{\gamma} \, G_{\nu}^{-|\gamma|/2} $ denote the Riesz transform associated with $G_\nu$. For $1 < p < \infty$, the $L^{p}(d\mu_{\nu})$-boundedness of these Riesz transforms was already known to be true when $\nu = 0$, and the same was shown to hold true for any $\nu \neq 0$ too in \cite{Garg-Garg-Riesz-transform-Grushin-drift-2026}. In fact, the $L^p(d\mu_{\chi})$-boundedness of Riesz transforms of arbitrary order associated with the sub-Laplacian with drift on certain amenable groups is due to Lohou\'{e}--Mustapha \cite{Lohoue-Mustapha-drift-2004}, from which the analogous result for $R_\gamma = X^{\gamma} \, G_{\nu}^{-|\gamma|/2} $ associated with $G_\nu$ was established in \cite{Garg-Garg-Riesz-transform-Grushin-drift-2026} by adapting the transference techniques from \cite{Transference_methods_in_analysis, Riesz_transform_Robinson}. 

\medskip 
Coming back to $G_{\nu}$, let $1< p< \infty$ and $\gamma, \, \nu$ be fixed. For any $c >0$, consider the local Riesz transforms $R_{\gamma, \, \nu, \, c}$, defined by 
\begin{align} 
\label{local-Riesz} 
R_{\gamma, \nu, c} := X^{\gamma} (G_\nu + cI)^{-|\gamma|/2}.
\end{align} 
Since $R_{\gamma, \nu, c} = R_{\gamma} \, G_\nu^{|\gamma|/2} (G_\nu + cI)^{-|\gamma|/2}$, it follows from the $L^{p}(d\mu_{\nu})$-boundedness of the Riesz transforms $R_{\gamma}$ and also of the operators $G_\nu^{|\gamma|/2} (G_\nu + cI)^{-|\gamma|/2}$ (see Remark \ref{rem:1}) that all the local Riesz transforms $R_{\gamma, \nu, c}$ are bounded on $L^{p}(d\mu_{\nu})$ uniformly in $c>0$.
\end{remark} 


\section{Sobolev Spaces} 
\label{sec:properties-sobolev-spaces}
As mentioned in the introduction, for $1<p<\infty$ and $\alpha \geq 0$, we define the Sobolev space 
$$
L^p_\alpha(d\mu_{\nu}) := \{f \in L^p(d\mu_\nu): G_{\nu}^{\alpha/2}f \in L^p(d\mu_\nu)\},
$$
endowed with the norm 
\begin{align*} 
\|f\|_{L^p_\alpha(d\mu_{\nu})} := \|f\|_{L^p(d\mu_\nu)}+\|G_{\nu}^{\alpha/2}f\|_{L^p(d\mu_\nu)}.
\end{align*} 
Note that for $\alpha = 0$, $L^{p}_{\alpha}(d\mu_{\nu})$ is nothing but $L^{p}(d\mu_{\nu})$. 

\medskip 
By using $L^p(d\mu_\nu)$-boundedness of operators $(G_\nu+cI)^{-\alpha/2}$ and $G_\nu^{\alpha/2} (G_\nu+cI)^{-\alpha/2}$ (see Remark \ref{rem:1}), one can show the following equivalent condition for the norm on $L^{p}_{\alpha}(d\mu_{\nu})$. 

\begin{lemma} 
\label{lem:Equivalence_of_norms_sobolev_space} 
Let $1<p<\infty$ and $\alpha > 0$. Given any $c>0$, we have 
\begin{equation} \label{eq:Equivalence_of_norms_sobolev_space}
\|f\|_{L^p_\alpha(d\mu_{\nu})}\sim \|(G_\nu+cI)^{\alpha/2} f\|_{L^p(d\mu_\nu)}.
\end{equation}
\end{lemma}
\begin{proof}
Observe first that 
\begin{align*}
\left\| f \right\|_{L^p_\alpha(d\mu_{\nu})} 
&= \|f\|_{L^p(d\mu_\nu)}+\|G_{\nu}^{\alpha/2}f\|_{L^p(d\mu_\nu)} \\ 
&= \left\| (G_\nu+cI)^{-\alpha/2} \left( (G_\nu+cI)^{\alpha/2} f \right) \right\|_{L^p(d\mu_\nu)} \\ 
&\quad + \left\| \left( G_{\nu}^{\alpha/2} (G_\nu+cI)^{-\alpha/2} \right) \left( (G_\nu+cI)^{\alpha/2} f \right) \right\|_{L^p(d\mu_\nu)} \\ 
&\lesssim \left\| (G_\nu+cI)^{\alpha/2} f \right\|_{L^p(d\mu_\nu)}.
\end{align*}

Next, we prove the reverse inequality. For $\alpha > 0$, there is a $k \in \mathbb{N}\cup \{0\}$ such that $\alpha \in (2k, 2k+2]$. Now, given any $g \in L^{p'}(d\mu_\nu)$, we can write 
\begin{align*}
\left| \left( (G_\nu+cI)^{\alpha/2} f, \, g \right) \right| & = \left| \left( (G_\nu+cI)^{k+1} f, \, (G_{\nu} + cI)^{-\left(k+1-\frac{\alpha}{2}\right)}g \right) \right| \\
& \lesssim \sum_{j=0}^{k+1} \left| \left(G_\nu^j f, \, (G_{\nu} + cI)^{-\left(k+1-\frac{\alpha}{2}\right)}g \right) \right|, 
\end{align*}
and therefore, we will be done if we can show that 
\begin{align*}
&\left| \left( G_\nu^j f, \, (G_{\nu} + cI)^{-\left(k+1-\frac{\alpha}{2}\right)}g \right) \right| \lesssim \left\| f \right\|_{L^p_\alpha(d\mu_{\nu})} \, \|g\|_{L^{p'}(d\mu_{\nu})}, 
\end{align*}
for every $0 \leq j \leq k+1$ and $g \in L^{p'}(d\mu_\nu)$, and let us now prove it as follows: 
\begin{align*}
&\left| \left( G_\nu^j f, \, (G_{\nu} + cI)^{-\left(k+1-\frac{\alpha}{2}\right)}g \right) \right| \\
& = \left| \left( G_\nu^j (G_{\nu} + cI)^{k+1} f, \, (G_{\nu} + cI)^{-\left(2k+2-\frac{\alpha}{2}\right)} g \right) \right| \\ 
&\lesssim \sum_{l=j}^{j+k+1} \left| \left( G_\nu^l f, \, (G_{\nu} + cI)^{-\left(2k+2-\frac{\alpha}{2}\right)} g \right) \right| \\ 
&= \sum_{l=j}^{k+1} \left| \left( f, \, G_\nu^l  (G_{\nu} + cI)^{-\left(2k+2-\frac{\alpha}{2}\right)} g \right) \right| + \sum_{l=k+2}^{j+k+1} \left| \left( G_\nu^{\alpha/2} f, \, G_\nu^{l-\frac{\alpha}{2}} (G_{\nu} + cI)^{-\left(2k+2-\frac{\alpha}{2}\right)} g \right) \right| \\
&\leq \sum_{l=j}^{k+1} \left\| f \right\|_{L^p(d\mu_\nu)} \left\| G_\nu^l (G_{\nu} + cI)^{-\left(2k+2-\frac{\alpha}{2}\right)} g \right\|_{L^{p'}(d\mu_\nu)} \\ 
&\quad + \sum_{l=k+2}^{j+k+1} \left\| G_\nu^{\alpha/2} f \right\|_{L^p(d\mu_\nu)} \left\| G_\nu^{l-\frac{\alpha}{2}} (G_{\nu} + cI)^{-\left(2k+2-\frac{\alpha}{2}\right)} g \right\|_{L^{p'}(d\mu_\nu)} \\
& \lesssim \left\| f \right\|_{L^p(d\mu_\nu)} \, \|g\|_{L^{p'}(d\mu_{\nu})} 
+ \left\| G_\nu^{\alpha/2} f \right\|_{L^p(d\mu_\nu)} \, \|g\|_{L^{p'}(d\mu_{\nu})} \\ 
& = \left\| f \right\|_{L^p_\alpha(d\mu_{\nu})} \, \|g\|_{L^{p'}(d\mu_{\nu})}, 
\end{align*}
and this completes the proof of Lemma \ref{lem:Equivalence_of_norms_sobolev_space}. 
\end{proof}

\begin{remark} 
\label{rem:boundedness-potential-operators}
As an immediate consequence of Remark \ref{rem:1} and Lemma \ref{lem:Equivalence_of_norms_sobolev_space}, we get the following results. 
\begin{enumerate}[(i)]
\item Sobolev spaces satisfy the natural embedding 
\begin{align*}
L^{p}_{\alpha_2}(d\mu_{\nu}) \hookrightarrow L^{p}_{\alpha_1}(d\mu_{\nu}) 
\end{align*} 
for any $0 \leq \alpha_1 < \alpha_2$.

\medskip 
\item 
We had seen in Remark \ref{rem:1} that given any $\alpha > 0$, the operators $(G_{\nu}+cI)^{-\alpha}$ and $G_{\nu}^{\alpha} (G_{\nu}+cI)^{-\alpha}$ are bounded on $L^{p}(d\mu_{\nu})$-spaces. It turns out that these operators are also bounded on Sobolev spaces $L^{p}_{\beta}(d\mu_{\nu})$ for all $\beta > 0$. In fact,
\begin{align*}
\|(G_{\nu}+cI)^{-\alpha}f\|_{L^{p}_{\beta}(d\mu_{\nu})} 
&\sim \|(G_{\nu}+cI)^{\beta/2}(G_{\nu}+cI)^{-\alpha}f\|_{L^{p}(d\mu_{\nu})} \\
&= \|(G_{\nu}+cI)^{-\alpha}(G_{\nu}+cI)^{\beta/2}f\|_{L^{p}(d\mu_{\nu})} \\
&\lesssim \|(G_{\nu}+cI)^{\beta/2} f\|_{L^{p}(d\mu_{\nu})} 
\\
&\sim \|f\|_{L^{p}_{\beta}(d\mu_{\nu})}. 
\end{align*}
In a similar manner, one can verify the $L^{p}_{\beta}(d\mu_{\nu})$-boundedness of $G_{\nu}^{\alpha}(G_{\nu}+cI)^{-\alpha}$. 
\end{enumerate}
\end{remark} 

These Sobolev spaces respect the following complex interpolation property. 
\begin{lemma} \label{lem:interpolation}
Let $1 < p < \infty$ and $\alpha_1, \, \alpha_2 \geq 0$. For any $\theta \in (0,1)$, we have
$$
(L^p_{\alpha_1}(d\mu_\nu), L^p_{\alpha_2}(d\mu_\nu))_{[\theta]} = L^p_{\alpha}(d\mu_\nu),
$$
where $\alpha = \theta \, \alpha_1 + (1-\theta) \, \alpha_2.$
\end{lemma} 

The proof of the above lemma can be developed following that of Theorem \cite[p. 103]{Interpolation_theory_Triebel}. In doing so, one would require the positivity of $G_{\nu}+cI$ for all $c>0$ on $L^{p}(d\mu_{\nu})$ spaces which is known from \cite[p. 91]{Interpolation_theory_Triebel}, as well as the $L^{p}(d\mu_{\nu})$-boundedness of the imaginary powers $G_{\nu}^{iu}$ which follows from \cite[p. 2203]{g_function}. With that, we omit the proof of this lemma which can easily be verified. 

\medskip 
We now discuss several basic properties of the Sobolev norm. Most of the proofs follow from those in \cite{Sobolev_embedding_Lie_groups} with minor modifications. But, to keep the work self-contained, we shall write some parts where there are some technical details owing to the Grushin structure. 

\begin{proposition} 
\label{prop:natural-power-norm-equi}
For any $\alpha \in \mathbb{N}$ and $p \in (1, \infty)$, we have 
$$
\|f\|_{L^p_{\alpha}(d\mu_\nu)} \sim \sum_{|\gamma| \leq \alpha} \|X^{\gamma} f\|_{L^p(d\mu_\nu)}.
$$
\end{proposition}
\begin{proof}
Using identity \eqref{eq:Equivalence_of_norms_sobolev_space}, it is enough to prove that 
\begin{align} 
\label{identity1-prop:natural-power-norm-equi}
\|(G_{\nu} + cI)^{\alpha/2} f\|_{L^p(d\mu_\nu)} \sim \sum_{|\gamma| \leq \alpha} \|X^{\gamma} f\|_{L^p(d\mu_\nu)}.
\end{align}

We begin with proving $\gtrsim$ in \eqref{identity1-prop:natural-power-norm-equi}. For the same, take and fix a multi-index $\gamma \in  \left(\mathbb{N} \cup \{0\} \right)^{d_1+d_1d_2}$ such that $|\gamma| \leq \alpha$. Now, using the boundedness of the local Riesz transforms and also of the operators $(G_{\nu} + cI)^{-\alpha}$ for any $\alpha > 0$, we have 
$$
\|X^{\gamma} (G_{\nu} + cI)^{-\alpha/2} f\|_{L^{p}(d\mu_{\nu})} = \|X^{\gamma} (G_{\nu} + cI)^{-|\gamma|/2} (G_{\nu} + cI)^{-(\alpha - |\gamma|)/2} f\|_{L^{p}(d\mu_{\nu})} \lesssim \|f\|_{L^{p}(d\mu_{\nu})},
$$
and equivalently
$$
\|X^{\gamma} f\|_{L^{p}(d\mu_{\nu})} \lesssim  \|(G_{\nu} + cI)^{\alpha/2} f\|_{L^{p}(d\mu_{\nu})}, 
$$
which completes the proof of $\gtrsim$ in \eqref{identity1-prop:natural-power-norm-equi}. 

\medskip 
For the other side inequality in \eqref{identity1-prop:natural-power-norm-equi}, note first that it is obvious when $\alpha$ is a positive even integer. Next, let us consider the case when $\alpha=1$. In the following computations, we shall also make use of the fact that with respect to the measure $d\mu_{\nu}$, we have $X_{j}^{\ast} = -X_{j} - 2\nu_{j}$ and $X_{j,k}^{\ast} = -X_{j,k}$. Now, 
\begin{align*}
& \left\|(G_{\nu} + cI)^{1/2}f \right\|_{L^{p}(d\mu_{\nu})} \\
& = \sup \left\{|((G_{\nu} + cI)^{1/2}f, \, g )| : \|g\|_{L^{p'}(d\mu_{\nu})} = 1\right\}  \\
&  = \sup \left\{|((G_{\nu} + cI) f, \, (G_{\nu} + cI)^{-1/2} g )| : \|g\|_{L^{p'}(d\mu_{\nu})} = 1\right\} \\
& \lesssim \|f\|_{L^{p}(d\mu_{\nu})} + \sum_{j=1}^{d_1} \sup \left\{|( X_j \left(X_{j} f \right), \, (G_{\nu} + cI)^{-1/2} g )| : \|g\|_{L^{p'}(d\mu_{\nu})} = 1\right\} \\ 
& \quad + \sum_{j=1}^{d_1} \sum_{k=1}^{d_2} \sup \left\{|( X_{j,k} \left( X_{j,k} f \right), \, (G_{\nu} + cI)^{-1/2} g )| : \|g\|_{L^{p'}(d\mu_{\nu})} = 1\right\} \\
& \quad + 2 \sum_{j=1}^{d_1} |\nu_j| \, \sup \left\{|( X_{j} f, \, (G_{\nu} + cI)^{-1/2} g )| : \|g\|_{L^{p'}(d\mu_{\nu})} = 1\right\} \\
& = \|f\|_{L^{p}(d\mu_{\nu})} + \sum_{j=1}^{d_1} \sup \left\{|( X_{j} f, \, - (X_{j} + 2\nu_{j})(G_{\nu} + cI)^{-1/2} g )| : \|g\|_{L^{p'}(d\mu_{\nu})} = 1\right\}\\
& \quad + \sum_{j=1}^{d_1} \sum_{k=1}^{d_2} \sup \left\{|( X_{j,k} f, \, - X_{j,k}(G_{\nu} + cI)^{-1/2} g )| : \|g\|_{L^{p'}(d\mu_{\nu})} = 1\right\} \\
& \quad + 2 \sum_{j=1}^{d_1} |\nu_j| \, \sup \left\{|( f, \, - (X_{j} + 2\nu_{j})(G_{\nu} + cI)^{-1/2} g )| : \|g\|_{L^{p'}(d\mu_{\nu})} = 1\right\} \\
& \lesssim \|f\|_{L^{p}(d\mu_{\nu})} + \sum_{j=1}^{d_1} \|X_{j} f\|_{L^{p}(d\mu_{\nu})} + \sum_{j=1}^{d_1} \sum_{k=1}^{d_2} \|X_{j,k} f\|_{L^{p}(d\mu_{\nu})},
\end{align*}
where we used the boundedness of the local Riesz transforms and of the operators $(G_{\nu} + cI)^{-1/2}$ on $L^{p}(d\mu_{\nu})$-spaces (see Remarks \ref{rem:1} and \ref{rem:2}).

\medskip Following the case of $\alpha = 1$ and that of all positive even integers $\alpha$, we can complete the claim for all remaining positive odd integers $\alpha \geq 3$ as follows: 
\begin{align*}
\|(G_{\nu} + cI)^{\alpha/2}f\|_{L^{p}(d\mu_{\nu})} & = \|(G_{\nu} + cI)^{1/2}(G_{\nu} + cI)^{(\alpha-1)/2}f\|_{L^{p}(d\mu_{\nu})} \\
& \lesssim \|(G_{\nu} + cI)^{(\alpha-1)/2} f\|_{L^{p}(d\mu_{\nu})} + \sum_{j=1}^{d_1} \|X_{j} (G_{\nu} + cI)^{(\alpha-1)/2} f\|_{L^{p}(d\mu_{\nu})} \\
& \quad + \sum_{j=1}^{d_1} \sum_{k=1}^{d_2} \|X_{j,k} (G_{\nu} + cI)^{(\alpha-1)/2} f\|_{L^{p}(d\mu_{\nu})} \\
& \lesssim \sum_{|\gamma| \leq \alpha} \|X^{\gamma} f\|_{L^p(d\mu_\nu)},
\end{align*}
where the last step follows by observing that $\alpha-1$ is even, and this completes the proof.
\end{proof}


The following proposition establishes a Banach space isomorphism between the two spaces $L^{p}_{\alpha}(d\mu_{\nu})$ and $L^{p}_{\alpha}(dx)$. 
\begin{proposition} \label{prop:Relation_between_sobolev_norm_with_and_without_drift}
Let $p \in (1,\infty)$ and $\alpha \geq 0$. The operator 
$$
T_{p}: L^{p}_{\alpha}(d\mu_{\nu}) \to L^{p}_{\alpha}(dx),
$$ 
defined by 
$$T_{p}(f)(x) := e^{2\nu\cdot x'/p} f(x),$$ is a Banach space isomorphism.
\end{proposition} 
Following the proof of \cite[Proposition 3.5]{Sobolev_embedding_Lie_groups} verbatim, the above proposition can be proved first for integers $\alpha \geq 0$ with the help of Proposition \ref{prop:natural-power-norm-equi} and then for arbitrary reals $\alpha \geq 0$ via interpolation (Lemma \ref{lem:interpolation}). We omit the details.  

\medskip 
Note that since $X_j$ and $X_{j,k}$ do not commute with $G_{\nu}$, the two different types of the Riesz transforms $X^{\gamma} G_{\nu}^{-|\gamma|/2}$ and $G_{\nu}^{-|\gamma|/2} X^{\gamma}$ may not behave in a similar manner. We shall show in a short while that $X^{\gamma} G_{\nu}^{-|\gamma|/2}$ are all $L^{p}_{\alpha}(d\mu_{\nu})$-bounded. At the moment, we do not know how to establish analogous boundedness properties of the other set of Riesz transforms $G_{\nu}^{-|\gamma|/2} X^{\gamma}$. But, the situation for the local Riesz transforms is relatively simple, and we have the following positive result. 

\begin{lemma} \label{lem:boundedness_of_local_Riesz_transform_on_sobolev_spaces-higher-order}
Let $1 < p < \infty, \, \alpha \geq 0$ and $c > 0$. The local Riesz transforms $\tilde{R}_{\gamma,\nu, c} = (G_{\nu} + cI)^{-|\gamma|/2} X^{\gamma}$ are all bounded on $L^{p}_{\alpha}(d\mu_{\nu})$. 
\end{lemma}
\begin{proof}
Let us first take the case when $\alpha = 0$. As already seen earlier, we have $X_{j}^{\ast} = -X_{j} - 2\nu_{j}$ and $X_{j,k}^{\ast} = -X_{j,k}$ for the measure $d\mu_{\nu}$. We shall also use the notation $X^* = (X_1^*, \ldots, X^*_{d_1}, \, X^*_{1,1}, \ldots, X^*_{d_1,d_2})$. For any $g \in L^{p'}(d\mu_{\nu})$, we have
$$
(\tilde{R}_{\gamma,\nu, c} f, g ) = ( (G_{\nu} + cI)^{-|\gamma|/2} X^{\gamma}, g ) = ( f , X^{\ast \gamma}(G_{\nu} + cI)^{-|\gamma|/2}g ),
$$
which implies that 
\begin{align*}
|(\tilde{R}_{\gamma,\nu, c} f, g)| & \leq \|f\|_{L^{p}(d\mu_{\nu})} \, \|X^{\ast \gamma}(G_{\nu} + cI)^{-|\gamma|/2}g\|_{L^{p'}(d\mu_{\nu})} \\
& \lesssim \|f\|_{L^{p}(d\mu_{\nu})} \sum_{\tilde{\gamma} \leq \gamma}\|X^{\tilde{\gamma}}(G_{\nu} + cI)^{-|\gamma|/2}g\|_{L^{p'}(d\mu_{\nu})} 
\lesssim \|f\|_{L^{p}(d\mu_{\nu})} \|g\|_{L^{p'}(d\mu_{\nu})}, 
\end{align*}
where we have used Remarks \ref{rem:1} and \ref{rem:2}, and hence
$$
\|\tilde{R}_{\gamma,\nu, c} f\|_{L^{p}(d\mu_{\nu})} \lesssim \|f\|_{L^{p}(d\mu_{\nu})}.
$$
Next, take arbitrary positive integer $\alpha$ such that $\alpha \geq |\gamma|$. For such an $\alpha$, we have 
\begin{align*}
\|\tilde{R}_{\gamma,\nu, c} f\|_{L^{p}_{\alpha}(d\mu_{\nu})} & =    
\|(G_{\nu} + cI)^{-|\gamma|/2} X^{\gamma} f\|_{L^{p}_{\alpha}(d\mu_{\nu})} \\
& \sim \|(G_{\nu} + cI)^{(\alpha -|\gamma|)/2} X^{\gamma} f\|_{L^{p}(d\mu_{\nu})} \\
& \sim \sum_{|\tilde{\gamma}| \leq \alpha - |\gamma|} \|X^{\tilde{\gamma}} (X^{\gamma} f)\|_{L^p(d\mu_\nu)} 
\lesssim \sum_{|\tilde{\gamma}| \leq \alpha} \|X^{\tilde{\gamma}} f\|_{L^p(d\mu_\nu)} 
\sim \|f\|_{L^{p}_{\alpha}(d\mu_\nu)}, 
\end{align*}
where the third and the last steps follow from Proposition \ref{prop:natural-power-norm-equi}. 

\medskip 
Finally, the result for any arbitrary real $\alpha > 0$ can be obtained via the standard interpolation techniques (using, for example, the complex interpolation from Lemma \ref{lem:interpolation}). 
\end{proof}


With the help of Lemma \ref{lem:boundedness_of_local_Riesz_transform_on_sobolev_spaces-higher-order}, we are able to prove the following recursive characterisation of Sobolev spaces. 

\begin{proposition}[Recursive characterisation] 
\label{prop:recursive_characterisation}
Let $p \in (1,\infty)$ and $\alpha \geq 0$, then $f \in L^p_{\alpha+1}(d\mu_\nu)$ if and only if $f, \, X_j f \, \text{and} \,  X_{j,k} f \in L^p_{\alpha}(d\mu_\nu)$, for all $1 \leq j \leq d_1$ and $1 \leq k \leq d_2$. Moreover, in that case, we have 
\begin{align}
\label{identity-prop:recursive_characterisation}
\|f\|_{L^p_{\alpha+1}(d\mu_\nu)} \sim \|f\|_{L^p_{\alpha}(d\mu_\nu)} + \sum_{j=1}^{d_1} \|X_j f\|_{L^p_{\alpha}(d\mu_\nu)}+ \sum_{j=1}^{d_1} \sum_{k=1}^{d_2} \|X_{j,k} f\|_{L^p_{\alpha}(d\mu_\nu)}.
\end{align} 
\end{proposition}
\begin{proof}
We begin with proving $\gtrsim$ in \eqref{identity-prop:recursive_characterisation}. 
In doing so, we shall use notation $\tilde{R}_{j, \, \nu, \, c}$ and $\tilde{R}_{j, \, k, \, \nu, \, c}$ for the first order local Riesz transforms $(G_{\nu}+cI)^{-1/2}X_{j}$ and $(G_{\nu}+cI)^{-1/2}X_{j, \, k}$ respectively. Now, note that 
\begin{align*}
& \|f\|_{L^p_{\alpha}(d\mu_\nu)} + \sum_{j=1}^{d_1} \|X_j f\|_{L^p_{\alpha}(d\mu_\nu)}+ \sum_{j=1}^{d_1} \sum_{k=1}^{d_2} \|X_{j,k} f\|_{L^p_{\alpha}(d\mu_\nu)} \\
& \lesssim \|f\|_{L^p_{\alpha+1}(d\mu_\nu)} + \sum_{j=1}^{d_1} \|(G_{\nu}+cI)^{\alpha/2}X_j f\|_{L^p(d\mu_\nu)}+ \sum_{j=1}^{d_1} \sum_{k=1}^{d_2} \|(G_{\nu}+cI)^{\alpha/2}X_{j,k} f\|_{L^p(d\mu_\nu)} \\
& \lesssim \|f\|_{L^p_{\alpha+1}(d\mu_\nu)} + \sum_{j=1}^{d_1} \|(G_{\nu}+cI)^{(\alpha+1)/2}\tilde{R}_{j,\nu, c} f\|_{L^p(d\mu_\nu)}+ \sum_{j=1}^{d_1} \sum_{k=1}^{d_2} \|(G_{\nu}+cI)^{(\alpha+1)/2}\tilde{R}_{j,k, \nu, c} f\|_{L^p(d\mu_\nu)} \\
& \lesssim \|f\|_{L^p_{\alpha+1}(d\mu_\nu)} + \sum_{j=1}^{d_1} \|\tilde{R}_{j, \nu, c} f\|_{L^p_{\alpha+1}(d\mu_\nu)}+ \sum_{j=1}^{d_1} \sum_{k=1}^{d_2} \|\tilde{R}_{j,k, \nu, c} f\|_{L^p_{\alpha+1}(d\mu_\nu)} \\  
& \lesssim \|f\|_{L^p_{\alpha+1}(d\mu_\nu)},
\end{align*}
where the last step follows from Lemma \ref{lem:boundedness_of_local_Riesz_transform_on_sobolev_spaces-higher-order}. 

\medskip 
Next, we prove $\lesssim$ in \eqref{identity-prop:recursive_characterisation}. Once again, using Lemma $\ref{lem:boundedness_of_local_Riesz_transform_on_sobolev_spaces-higher-order}$, we have
\begin{align*}
\|f\|_{L^p_{\alpha+1}(d\mu_\nu)} & \sim \|(G_{\nu}+cI)^{-1/2} \, (G_{\nu}+cI) f\|_{L^p_{\alpha}(d\mu_\nu)} \\
& \lesssim \sum_{j=1}^{d_1} \|(G_{\nu}+cI)^{-1/2} \, X_{j}^{2} f\|_{L^p_{\alpha}(d\mu_\nu)} + \sum_{j=1}^{d_1} \sum_{k=1}^{d_2} \|(G_{\nu}+cI)^{-1/2} \, X_{j,k}^{2} f\|_{L^p_{\alpha}(d\mu_\nu)} \\
& \quad + \sum_{j=1}^{d_1} \|(G_{\nu}+cI)^{-1/2} \,  X_{j} f\|_{L^p_{\alpha}(d\mu_\nu)} + \|(G_{\nu}+cI)^{-1/2} \,  f\|_{L^p_{\alpha}(d\mu_\nu)} \\
& \lesssim \sum_{j=1}^{d_1} \|X_{j} f\|_{L^p_{\alpha}(d\mu_\nu)} + \sum_{j=1}^{d_1} \sum_{k=1}^{d_2} \|X_{j,k} f\|_{L^p_{\alpha}(d\mu_\nu)} + \| f\|_{L^p_{\alpha}(d\mu_\nu)}, 
\end{align*}
which completes the proof of Proposition \ref{prop:recursive_characterisation}.
\end{proof}


As mentioned in the introduction, as a byproduct of the results proved thus far, we shall prove the Sobolev space boundedness of the Riesz transform of arbitrary order, that is, Theorem \ref{thm:Riesz-transform-boundedness-Sobolev}.
\begin{proof}[Proof of Theorem \ref{thm:Riesz-transform-boundedness-Sobolev}:]
Note first that for $\alpha = 0$, the space $L^{p}_{0}(d\mu_{\nu})$ is nothing but $L^{p}(d\mu_{\nu})$ and the boundedness of $R_{\gamma}$ on $L^p$-spaces is already known (see Remark \ref{rem:2}). Now, we show that given an $\alpha \geq 0$, if all the Riesz transforms $R_{\gamma}$ (for all multi-indices $\gamma$) are bounded on $L^{p}_{\alpha}(d\mu_{\nu})$, then they all are also bounded on $L^{p}_{\alpha+1}(d\mu_{\nu})$. Once we prove it, the result for arbitrary real $\alpha \geq 0$ would follow from the standard interpolation techniques. 
So, let us assume the boundedness of the Riesz transforms $R_{\gamma}$ on $L^{p}_{\alpha}(d\mu_{\nu})$ for all multi-indices, and then 
\begin{align*}
\|R_{\gamma}f\|_{L^{p}_{\alpha+1}(d\mu_{\nu})} & \sim \|R_{\gamma}f\|_{L^p_{\alpha}(d\mu_\nu)} + \sum_{j=1}^{d_1} \|X_j R_{\gamma}f\|_{L^p_{\alpha}(d\mu_\nu)}+ \sum_{j=1}^{d_1} \sum_{k=1}^{d_2} \|X_{j,k} R_{\gamma} f\|_{L^p_{\alpha}(d\mu_\nu)} \\
& \lesssim \|f\|_{L^p_{\alpha}(d\mu_\nu)} + \sum_{j=1}^{d_1} \|X_j X^{\gamma} \, G_{\nu}^{-|\gamma|/2}f\|_{L^p_{\alpha}(d\mu_\nu)}+ \sum_{j=1}^{d_1} \sum_{k=1}^{d_2} \|X_{j,k} X^{\gamma} \, G_{\nu}^{-|\gamma|/2} f\|_{L^p_{\alpha}(d\mu_\nu)} \\
& \lesssim \|f\|_{L^p_{\alpha}(d\mu_\nu)} + \sum_{j=1}^{d_1} \| G_{\nu}^{1/2}f\|_{L^p_{\alpha}(d\mu_\nu)}+ \sum_{j=1}^{d_1} \sum_{k=1}^{d_2} \| G_{\nu}^{1/2}f\|_{L^p_{\alpha}(d\mu_\nu)} \\
& \lesssim \|f\|_{L^p_{\alpha+1}(d\mu_\nu)},
\end{align*}
where we have made use of Proposition \ref{prop:recursive_characterisation} in the first step, whereas the last step follows from the definition of the Sobolev spaces and the embedding given in Remark \ref{rem:boundedness-potential-operators}. This completes the proof of the Theorem \ref{thm:Riesz-transform-boundedness-Sobolev}. 
\end{proof}

\begin{corollary} 
\label{corollary:1}
Let $1 < p < \infty$ and $\alpha \geq 0$. Given any multi-index $\gamma$ and $c>0$, the local Riesz transforms $R_{\gamma, \nu, c} = X^{\gamma} (G_{\nu} + cI)^{-|\gamma|/2}$ are bounded on $L^{p}_{\alpha}(d\mu_{\nu})$. 
\end{corollary}


\section{Embedding theorems} 
\label{sec:proof-embedding}
In this section, we shall prove our main embedding result, that is, Theorem \ref{thm:embedding}. We will make use of the following estimate of the integral kernel of $(G + cI)^{-\alpha/2}$.

\begin{lemma} \label{lemma:kernel_estimate_1}
Given $\alpha, \, c > 0$, let $K_{\alpha, \, c}$ denote the integral kernel of the operator $(G+cI)^{-\alpha/2}$. For every $p \in [1,\infty]$ and 
$\alpha > Q/p'$, we have
\begin{equation} \label{eq:kernel_estimate_1}
\sup_{x} |B(x,1)|^{1/p'}\left(\int_{\R^{d_1+d_2}} |K_{\alpha, \, c}(x,y)|^{p} \, dy\right)^{1/p} \leq C,
\end{equation}
for some constant $C>0$ and with usual modifications for $p=\infty$. 
\end{lemma}
\begin{proof}
We first prove the above estimate for $p = \infty$ and $\alpha > Q$. Note that in view of \eqref{integral-identity-negative-fractional-power} (which is also true for $\nu = 0$), we have
\begin{align*}
|K_{\alpha, \, c}(x,y)| & =  \frac{1}{\Gamma(\alpha/2)}\int_{0}^\infty t^{\alpha/2-1} \, e^{-ct} \, H_{t}(x,y) \, dt  \\
& \lesssim \int_{0}^\infty t^{\alpha/2-1} \, e^{-ct} \, |B(x,\sqrt{t})|^{-1} \, e^{-b'\rho(x,y)^2/t} \, dt  \\
& \lesssim |B(x,1)|^{-1}  \int_{0}^\infty t^{\alpha/2-1} \, e^{-ct} \, \max \left\{t^{-Q/2}, \, 1 \right\} dt  \\
& \lesssim |B(x,1)|^{-1}  \int_{0}^\infty t^{\frac{\alpha-Q}{2}-1} \, e^{-\frac{ct}{2}} \, dt, 
\end{align*}
where the last integral converges for any $\alpha > Q$. Equivalently, 
$$
\sup_{x, \, y} \, |B(x,1)| \, |K_{\alpha, \, c}(x,y)| \leq C,
$$
which is \eqref{eq:kernel_estimate_1} when $p=\infty.$

\medskip 
Let us now prove the remaining case, that is, when $1 \leq p < \infty$ and $\alpha > Q/p'$. Again, in view of \eqref{integral-identity-negative-fractional-power}, we have 
\begin{align*}
\left(\int_{\R^{d_1+d_2}} |K_{\alpha, \, c}(x,y)|^{p} \, dy\right)^{1/p} 
& = \left(\int_{\R^{d_1+d_2}} \left|\frac{1}{\Gamma(\alpha/2)} \int_{0}^\infty t^{\alpha/2-1} \, e^{-ct} \, H_{t}(x,y) \, dt \right|^{p}\, dy\right)^{1/p} \\
& \leq \int_{0}^\infty  |B(x, \sqrt{t})|^{-1} \, t^{\alpha/2-1} \, e^{-ct} \left(\int_{\R^{d_1+d_2}} e^{-p b'\rho(x,y)^{2}/t} \, dy \right)^{1/p} dt \\
& \lesssim \int_{0}^\infty |B(x, \sqrt{t})|^{-1/p'} \, t^{\alpha/2-1} e^{-c t} \, dt \\
& \lesssim |B(x, 1)|^{-1/p'} \, \int_{0}^\infty t^{\alpha/2-1} e^{-c t} \, \max \{t^{-Q/(2 p')}, 1\} \, dt \\ 
& \lesssim |B(x, 1)|^{-1/p'},
\end{align*} 
where the convergence in the last step holds true whenever $\alpha > Q/p'$. Hence, 
$$
\sup_{x} |B(x, 1)|^{1/p'} \left(\int_{\R^{d_1+d_2}} |K_{\alpha, \, c}(x,y)|^{p} \, dy\right)^{1/p} \leq C,
$$
which is \eqref{eq:kernel_estimate_1} for $1 \leq p < \infty$, and this completes the proof of the lemma. 
\end{proof}


\medskip 
We are now in a position to prove Theorem \ref{thm:embedding}.

\begin{proof}[Proof of Theorem \ref{thm:embedding}:]

We shall prove all parts of the theorem one by one. 

\medskip \noindent 
\textbf{\underline{Part (i)}:} Here, we shall prove that for any $1<p \leq q<\infty$ and $\frac{1}{p}-\frac{1}{q} \leq \frac{\alpha}{Q}$, we have
\begin{align} \label{part-1-theorem-embedding}
\|V_\nu(\cdot,1)^{\frac{1}{p}-\frac{1}{q}} \, f\|_{L^q(d\mu_\nu)} \lesssim \|f\|_{L^p_\alpha(d\mu_\nu)}.
\end{align}

We shall first prove \eqref{part-1-theorem-embedding} for $\nu = 0$. 

\medskip 
\textbf{\underline{Case I}:} When $p = q$. In this case, the above inequality is trivially true.

\medskip 
\textbf{\underline{Case II}:}  When $p < q$. The claimed inequality is known to be true when $\frac{1}{p}-\frac{1}{q} = \frac{\alpha}{Q}$ (see \cite[Theorem 2.6]{Bagchi-Basak-Garg-Ghosh-sparse-pseudo-multiplier-grushin-II-JGA-2024}), and from this the same holds true in the case when $\frac{1}{p}-\frac{1}{q} < \frac{\alpha}{Q}$ by invoking the embedding given in Remark \ref{rem:boundedness-potential-operators}. This completes the proof in the case of $\nu = 0$.

\medskip Now, let us work with $\nu \neq 0$. In this case, note that 
\begin{align*}
\|V_{\nu}(\cdot,1)^{\frac{1}{p}-\frac{1}{q}}f\|_{L^q(d\mu_\nu)} & = 
\left(\int_{\R^{d_1+d_2}} \left| V_{\nu}(x,1)^{\frac{1}{p}-\frac{1}{q}}f(x) \right|^q d\mu_\nu\right)^{1/q} \\
& \sim  \left( \int_{\R^{d_1+d_2}} \left| |B(x,1)|^{\frac{1}{p}-\frac{1}{q}} e^{2\nu\cdot x'\left(\frac{1}{p}-\frac{1}{q}\right)}f(x) e^{\frac{2\nu\cdot x'}{q}} \right|^q \, dx \right)^{1/q} \\
&= \left( \int_{\R^{d_1+d_2}} \left| |B(x,1)|^{\frac{1}{p}-\frac{1}{q}} T_{p}(f) (x) \right|^q \, dx \right)^{1/q} \\
& \lesssim \|T_{p}(f)\|_{L^p_\alpha(dx)} \\ 
& \sim \|f\|_{L^p_\alpha(d\mu_\nu)}, 
\end{align*}
where the second last step follows from the case when $\nu = 0$ and the final step follows from Proposition \ref{prop:Relation_between_sobolev_norm_with_and_without_drift}. This completes the proof of inequality \eqref{part-1-theorem-embedding}. 

\medskip \noindent 
\textbf{\underline{Part (ii)}:} 
We shall prove that for $1<p<\infty$ and $\alpha>Q/p,$ we have
\begin{align} \label{part-2-theorem-embedding}
\|V_\nu(\cdot,1)^{\frac{1}{p}}f\|_{L^\infty(d\mu_{\nu})} \lesssim \|f\|_{L^p_\alpha(d\mu_\nu)}.
\end{align}

Let us first prove it for $\nu = 0$. For that, define $g = (G + cI)^{\alpha/2} \, f$, and then using Lemma \ref{lemma:kernel_estimate_1}, we get
\begin{align*}
|f(x)| =  |(G + cI)^{-\alpha/2} \, g(x)| 
& = \left|\int_{\R^{d_1+d_2}} K_{\alpha, \, c} (x,y) \, g(y) \, dy \right| \\
& \leq \|g\|_{L^{p}(dx)} \, \left(\int_{\R^{d_1+d_2}} |K_{\alpha, \, c}(x,y)|^{p'} \, dy \right)^{1/p'} \\
& \lesssim \|f\|_{L^{p}_{\alpha}(dx)} \, |B(x,1)|^{-1/p},
\end{align*}
which implies that 
$$
\||B(\cdot, 1)|^{1/p} \, f\|_{L^{\infty}(dx)} \lesssim \|f\|_{L^{p}_{\alpha}(dx)}.
$$

Next, let us work with an arbitrary drift vector $\nu \neq 0.$ Take and fix large enough $c$ satisfying $ c > \max\{\frac{1}{4b'}, \, \frac{4 |\nu|^2}{b'} \left| \frac{1}{p} - \frac{1}{p'} \right|^2\}$ and define $g = (G_{\nu}+cI)^{\alpha/2}f$. Let $K_{\alpha, \, c}^{\nu}$ denote the integral kernel of the operator $(G_{\nu}+cI)^{-\alpha/2}$.
Then, 
\begin{align*}
|e^{2\nu\cdot x'/p} \, f(x)| & = |e^{2\nu\cdot x'/p}  \, (G_\nu+cI)^{-\alpha/2} \, g(x)| \\ 
& = \left|\int_{\R^{d_1+d_2}} e^{2\nu\cdot x'/p} \, K_{\alpha, \, c}^{\nu}(x,y) \, g(y) \, d\mu_{\nu}(y)\right| \\ 
& \leq \|g\|_{L^p(d\mu_\nu)} \left(\int_{\R^{d_1+d_2}} |K_{\alpha, c}^{\nu}(x,y)|^{p'} e^{2\nu\cdot x'p'/p}  \, d\mu_{\nu}(y)\right)^{1/p'}
\\ & = \|g\|_{L^p(d\mu_\nu)}(I_1(x) + I_2(x))^{1/p'},
\end{align*}
where $I_1(x)$ and $I_2(x)$ denote the integrals in the second last step over $\rho(x,y) \geq 1$ and $\rho(x,y) < 1$ respectively.

\medskip We shall first consider $I_1(x)$. In doing so, we also make use of the 9th identity of $\S 3.471$ of \cite{Tables_of_integrals_book} and the asymptotes of the modified Bessel function to conclude that 
$$ 
|K_{\alpha, \, c}^{\nu}(x,y)| \lesssim |B(x,1)|^{-1} e^{-\nu \cdot (x' + y')} \, e^{-\sqrt{b'c} \, \rho(x,y)},
$$ 
for $\rho(x,y) \geq 1.$ Consequently, 
\begin{align*}
I_1(x) & = \int_{\rho(x,y) \geq 1} e^{2\nu\cdot x' p'/p} |K_{\alpha, \, c}^{\nu}(x,y)|^{p'} e^{2\nu\cdot y'} \, dy \\
& \lesssim \int_{\rho(x,y)\geq 1} e^{2\nu\cdot x' p'/p} |B(x,1)|^{-p'} e^{-\nu \cdot (x' + y') p'} e^{-p'\sqrt{b'c}\rho(x,y)} e^{2\nu\cdot y'} \, dy \\
& = |B(x,1)|^{-p'} \int_{\rho(x,y)\geq 1} e^{\nu \cdot (x'-y') \left( \frac{p'}{p} - 1 \right)} e^{-p'\sqrt{b'c}\rho(x,y)}  \, dy \\
& = |B(x,1)|^{-p'} \sum_{k=0}^{\infty}\int_{2^k \leq \rho(x,y)<2^{k+1}} e^{\nu \cdot (x'-y') \left( \frac{p'}{p} - 1 \right)} e^{-p'\sqrt{b'c}\rho(x,y)}  \, dy \\ 
& \lesssim |B(x,1)|^{-p'} \sum_{k=0}^{\infty}\int_{2^k \leq \rho(x,y)<2^{k+1}} e^{-p'\sqrt{b'c} 2^{k-1}}  \, dy \\
& \lesssim |B(x,1)|^{-p'} \sum_{k=0}^{\infty} e^{-p'\sqrt{b'c} 2^{k-1}} |B(x,2^{k+1})| \\
& \lesssim |B(x,1)|^{-p'} \sum_{k=0}^{\infty} e^{-p'\sqrt{b'c}2^{k-1}} 2^{(k+1)Q} |B(x,1)| \\
& \lesssim |B(x,1)|^{-p'/p}.
\end{align*}

\medskip Next, we consider $I_2(x)$. We shall use the fact that $|K_{\alpha, \, c}^{\nu}(x,y)| \leq e^{-\nu \cdot (x'+ y')} \, |K_{\alpha, \, c}(x,y)|$, where $K_{\alpha, \, c}$ denotes the integral kernel of the operator $(G+cI)^{-\alpha/2}$ as in Lemma \ref{lemma:kernel_estimate_1}.
\begin{align*}
I_2(x) &= \int_{\rho(x,y) < 1} e^{2\nu\cdot x' p'/p} \, |K_{\alpha, \, c}^{\nu}(x,y)|^{p'} \, e^{2\nu\cdot y'} \, dy \\  
& \leq \int_{\rho(x,y)< 1} e^{2\nu\cdot x' p'/p}  e^{-\nu\cdot (x' + y') p'} \, |K_{\alpha, \, c}(x,y)|^{p'} \, e^{2\nu\cdot y'} \, dy \\
& = \int_{\rho(x,y)< 1} |K_{\alpha, \, c}(x,y)|^{p'} \, e^{\nu \cdot (x'-y') \left( \frac{p'}{p} - 1 \right)} \, dy\\ 
& \lesssim \int_{\rho(x,y)< 1} |K_{\alpha, \, c}(x,y)|^{p'} \, dy \\
& \leq |B(x,1)|^{-p'/p}, 
\end{align*}
where the last step follows from Lemma \ref{lemma:kernel_estimate_1}. 

\medskip Put together, we have shown that 
$$
|e^{2\nu\cdot x'/p} f(x)| \lesssim \|g\|_{L^{p}(d\mu_{\nu})} \, |B(x,1)|^{-1/p},
$$
which immediately implies the claimed estimate \eqref{part-2-theorem-embedding} and this completes the proof of part (ii) of Theorem \ref{thm:embedding}.

\medskip \noindent 
\textbf{\underline{Part (iii)}:} 
We shall prove that for $1 < p < \infty, \, 1 \leq q \leq \infty, \, \nu \neq 0$ and $\eta \geq 0$, if 
\begin{align}
\label{part-3-theorem-embedding}
\|V_\nu(\cdot,1)^{\eta} \, f\|_{L^q(d\mu_\nu)} \lesssim \|f\|_{L^p_\alpha(d\mu_\nu)},
\end{align}
for all $f \in L^p_\alpha(d\mu_\nu)$, then we must have $\eta = \frac{1}{p} - \frac{1}{q}$. 

\medskip 
We shall prove it via the transference technique, and for the same we take inputs from our work done in \cite[Section 4]{Garg-Garg-Riesz-transform-Grushin-drift-2026}. Let us begin with recalling the following. On unimodular Lie groups, working with the sub-Laplacian with drift $\Delta_\chi$ (see identity \eqref{def:sub-Lap-with-drift} in the introduction), it is known from \cite{Sobolev_embedding_Lie_groups} that an inequality of the form 
\begin{align}
\label{ineq:contrary-inequality-group-drift}
\|\mu_{\chi}(B(\cdot, 1))^{\eta} \, f\|_{L^{q}(d\mu_{\chi})} \lesssim \|f\|_{L^{p}_{\alpha}(d\mu_{\chi})}, 
\end{align}
can hold only if $\eta = \frac{1}{p}-\frac{1}{q}$. Note that in \cite{Sobolev_embedding_Lie_groups}, it was only argued that the above embedding fails when $p \neq q$ and $\eta = 0$, however it is easy to verify that by minor modifications in their arguments one can conclude the more general statement as mentioned above. 

\medskip 
For technical convenience, let us assume here onwards that $\nu = e_1$, and it will be clear from the proof that similar arguments are true for any arbitrary $\nu \neq 0$ by choosing appropriate parameters. Let us fix $\xi = (\xi',\xi'') \in \R^{d_1+d_2}$ such that $|\xi'|=1$ (to be chosen later) and let $U: L^2(\R^{d_1+d_2}) \to L^2(\R^{d_1+d_2})$ be the Euclidean translation by $\xi$, that is, 
$$
U f(x',x'') = f(x'+\xi', x''+\xi''),
$$ 
so that $U^{-1} f(x',x'')= f(x'-\xi', x''-\xi'').$ Also, for any $R>0,$ define $\Lambda_R:  L^2(\R^{d_1+d_2}) \to L^2(\R^{d_1+d_2})$ to be the isotropic Euclidean dilation, that is, 
$$
\Lambda_R f(x', x'')= f(Rx', Rx''),
$$ 
so that $\Lambda_R^{-1} = \Lambda_{R^{-1}}.$

\medskip 
On the Euclidean space $\R^{d_1+d_2}$, by abuse of notation, let us denote by $e_1$ the usual unit vector and by $\Delta_{e_1}$ the corresponding Laplacian with drift. We claim that for any $c>0$, for $C_c^\infty$ functions, we have 
\begin{equation} \label{eq:limit_giving_neg_result}
\lim_{R\to 0} \, (\Lambda_R U)\left(R^2G_{\frac{e_1}{R}}+cI\right)^{-\alpha/2} (U^{-1} \Lambda_R^{-1} f)(x', x'') = (\Delta_{e_1}+cI)^{-\alpha/2}f(x).  
\end{equation}

Assuming \eqref{eq:limit_giving_neg_result} for now, let us first prove our main result by contradiction. That is, let us assume that for some $0 \leq \eta \neq \frac{1}{p} - \frac{1}{q}$, we have 
$$
\|V_{e_1}(\cdot, 1)^{\eta} \, f\|_{L^q(d\mu_{e_1})} \lesssim  \|(G_{e_1}+cI)^{\alpha/2}f\|_{L^p(d\mu_{e_1})}, 
$$
or equivalently, 
\begin{align}
\label{ineq:contrary-inequality-grushin-drift}
\|V_{e_1}(\cdot, 1)^{\eta} (G_{e_1}+cI)^{-\alpha/2}f\|_{L^q(d\mu_{e_1})} \lesssim \|f\|_{L^p(d\mu_{e_1})}.
\end{align}

\medskip 
We know that $H_{t,\, se_1}(x,y) = s^{d_1+2d_2} \, H_{s^2t, \, e_1}(\delta_s x, \delta_s y)$ for any $s > 0$ (see \cite[Subsection 2.5.2]{Garg-Garg-Riesz-transform-Grushin-drift-2026}), using which one can easily verify that 
\begin{align}
\label{identity:scaling-grushin-drift-fractional-power}
\left(R^2G_{\frac{e_1}{R}}+cI\right)^{-\alpha/2} = \delta_R^{-1} \, (G_{e_1}+cI)^{-\alpha/2} \, \delta_R. 
\end{align}

\medskip Making use of identity \eqref{identity:scaling-grushin-drift-fractional-power}, we get 
\begin{align*} 
& \left(\int_{\R^{d_1+d_2}}\left|(\Lambda_R U)\left(R^2G_{\frac{e_1}{R}}+cI\right)^{-\alpha/2} (U^{-1}\Lambda_R^{-1}f)(x', x'') \right|^q  e^{2x_1'(1+\eta q)} \, dx \right)^{1/q} \\
& \quad = \left(\int_{\R^{d_1+d_2}}\left|(\Lambda_R \, U \, \delta_R^{-1})\left(G_{e_1}+cI\right)^{-\alpha/2}(\delta_R \, U^{-1} \, \Lambda_R^{-1}f)(x', x'')\right|^q  e^{2x_1'(1+\eta q)} \, dx \right)^{1/q} \\
& \quad \leq e^{\frac{-2\xi_1'\eta}{R}}\left(\int_{\R^{d_1+d_2}}\left|(\Lambda_R \, U \, \delta_R^{-1})(V_{e_1}(\cdot, 1)^{\eta}\left(G_{e_1}+cI\right)^{-\alpha/2}(\delta_R \, U^{-1} \, \Lambda_R^{-1}f))(x', x'')\right|^q  e^{2x_1'} \, dx \right)^{1/q} \\
& \quad = e^{\frac{-2\xi_1'\eta}{R}} \, e^{\frac{-2\xi_1'}{Rq}} R^{d_2/q}\left(\int_{\R^{d_1+d_2}}\left|V_{e_1}(\cdot, 1)^{\eta} \left(G_{e_1}+cI\right)^{-\alpha/2}(\delta_R \, U^{-1} \, \Lambda_R^{-1}f)\left(x',x''\right)\right|^q e^{2x_1'} \, dx \right)^{1/q} \\
& \quad \lesssim e^{\frac{-2\xi_1'\eta}{R}} \, e^{\frac{-2\xi_1'}{Rq}} R^{d_2/q} \left( \int_{\R^{d_1+d_2}} \left| (\delta_R \, U^{-1} \, \Lambda_R^{-1}f)\left(x',x''\right)\right|^p e^{2x_1'} \, dx \right)^{1/p} \\ 
& \quad = e^{\frac{-2\xi_1'\eta}{R}} \, e^{\frac{-2\xi_1'}{Rq}} R^{d_2/q}\left(\int_{\R^{d_1+d_2}}\left|f\left(x'-\frac{\xi'}{R},Rx''-\frac{\xi''}{R}\right)\right|^p  e^{2x_1'} \, dx \right)^{1/p} \\
& \quad = e^{\frac{2\xi_1'}{R}\left(\frac{1}{p}-\frac{1}{q}-\eta\right)} R^{d_2\left(\frac{1}{q}-\frac{1}{p}\right)} \left(\int_{\R^{d_1+d_2}}\left|f\left(x',x''\right)\right|^p  e^{2x_1'} \, dx \right)^{1/p}.
\end{align*}

Now, there are two possibilities: 
\begin{itemize}
\item When $\frac{1}{p} - \frac{1}{q} - \eta > 0$. In this case, we work with $\xi' = -e_1$, 

\item When $\frac{1}{p} - \frac{1}{q} - \eta < 0$. In this case, we work with $\xi' = +e_1$. 
\end{itemize} 

With the above choice of $\xi'$, it is easy to observe that in either of the two cases, one has 
$$
e^{\frac{2\xi_1'}{R}\left(\frac{1}{p}-\frac{1}{q}-\eta\right)} R^{d_2\left(\frac{1}{q}-\frac{1}{p}\right)} \lesssim_{p, \,q} 1,
$$ 
uniformly over $0 < R \leq 1$. 

\medskip 
Summarising, we have that on $0 < R \leq 1$, 
\begin{align}
\label{ineq:dilated-contradiction-1}
\left(\int_{\R^{d_1+d_2}}\left|(\Lambda_R U)\left(R^2G_{\frac{e_1}{R}}+cI\right)^{-\alpha/2} (U^{-1}\Lambda_R^{-1}f)(x' , x'') \right|^q  e^{2x_1'(1+\eta q)} \, dx \right)^{1/q} \lesssim \|f\|_{L^p(d\mu_{e_1})}.
\end{align}

As a consequence, for the character $\chi(x) = e^{2x_1'}$ on $\R^{d_1+d_2}$, we get 
\begin{align*} 
& \|\mu_{\chi}(B(\cdot, 1))^{\eta}(\Delta_{e_1}+cI)^{-\alpha/2}f\|_{L^q(d\mu_{\chi})} \\
& = \left(\int_{\R^{d_1+d_2}}\left|e^{2x_1'\eta}(\Delta_{e_1}+cI)^{-\alpha/2}f(x)\right|^q e^{2x_1'} \, dx\right)^{1/q} \\ 
& = 
\left(\int_{\R^{d_1+d_2}}\left|\lim_{R\to 0} \, (\Lambda_R U) \left(R^2G_{\frac{e_1}{R}}+cI\right)^{-\alpha/2} (U^{-1}\Lambda_R^{-1}f)(x', x'') \right|^q  e^{2x_1'(1+\eta q)} \, dx \right)^{1/q}\\ 
& \lesssim \liminf_{R \to 0} \, \left(\int_{\R^{d_1+d_2}}\left|(\Lambda_R U) \, \left(R^2G_{\frac{e_1}{R}}+cI\right)^{-\alpha/2} (U^{-1}\Lambda_R^{-1}f)(x', x'') \right|^q  e^{2x_1'(1+\eta q)} \, dx \right)^{1/q} \\
& \lesssim \|f\|_{L^p(d\mu_{e_1})} = \|f\|_{L^p(d\mu_{\chi})},
\end{align*}
where the second step follows from the pointwise convergence \eqref{eq:limit_giving_neg_result}, and the second last step follows from \eqref{ineq:dilated-contradiction-1}. 

\medskip 
But, the concluded embedding contradicts \eqref{ineq:contrary-inequality-group-drift}, and hence \eqref{ineq:contrary-inequality-grushin-drift} is not possible for $0 \leq \eta \neq \frac{1}{p} - \frac{1}{q}$. 

\medskip It only remains to verify the pointwise convergence as claimed in \eqref{eq:limit_giving_neg_result}. Take and fix a $C_c^\infty$ function $f$ and set $F(x) = f(x) \, e^{x_1'}$. Following the expression derived right after (4.2) in \cite{Garg-Garg-Riesz-transform-Grushin-drift-2026}, we have 
\begin{align*} 
& (\Lambda_R U)e^{-tR^2G_{e_1/R}} (U^{-1} \Lambda_R^{-1} f)(x', x'') \\ 
& \quad = (2 \pi)^{-d_2} \int_{\R^{d_1+d_2}} e^{-i\lambda'' \cdot x''} \, e^{-t} \, e^{-x_1'} \left(\cosh{(2tR|\lambda''|)}\right)^{-d_1/2} e^{2\pi i x'\cdot \lambda'} \\ 
& \qquad \exp{\left( \frac{-|\lambda''|\tanh{(tR|\lambda''|)}}{R} \left(1 - \frac{\tanh{(tR|\lambda''|)}}{2 \, \coth{(2tR|\lambda''|)}}\right) |Rx' + \xi'|^2 \right)} \, \widehat{F}\left(\lambda',\frac{-\lambda''}{2\pi}\right) \\ 
& \qquad \exp{\left(\frac{-2\pi^2|\lambda'|^2}{R|\lambda''|\coth{(2tR|\lambda''|)}}\right)} \exp{\left(-2\pi i \frac{\tanh{(tR|\lambda''|)}}{R \, \coth{(2tR|\lambda''|)}} \lambda'\cdot \left(Rx'+ \xi'\right)\right)} \, d\lambda,
\end{align*}
which implies that 
\begin{align} \label{eq:limit_step}
& (\Lambda_R U) \left(R^2G_{\frac{e_1}{R}}+cI\right)^{-\alpha/2} (U^{-1}\Lambda_R^{-1}f)(x', x'') \\
&\nonumber = \frac{1}{\Gamma(\alpha/2)} \int_{0}^\infty t^{\frac{\alpha}{2}-1} e^{-ct} (\Lambda_R U)e^{-t R^2 G_{e_1/R}} (U^{-1}\Lambda_R^{-1}f)(x', x'') \, dt  \\
&\nonumber = \frac{(2\pi)^{-d_2}}{\Gamma(\alpha/2)} \int_{0}^\infty t^{\frac{\alpha}{2}-1} e^{-ct} \int_{\R^{d_1+d_2}} e^{-i\lambda'' \cdot x''} e^{-t} e^{-x_1'} \left(\cosh{(2tR|\lambda|)}\right)^{-d_1/2} e^{2\pi i x'\cdot \lambda'} 
\\ & \nonumber \quad \exp{\left( \frac{-|\lambda''|\tanh{(tR|\lambda''|)}}{R} \left(1 - \frac{\tanh{(tR|\lambda''|)}}{2 \, \coth{(2tR|\lambda''|)}}\right) |Rx' + \xi'|^2 \right)} \, \widehat{F}\left(\lambda',\frac{-\lambda''}{2\pi}\right) 
\\ & \nonumber \quad \exp{\left(\frac{-2\pi^2|\lambda'|^2}{R|\lambda''|\coth{(2tR|\lambda''|)}}\right)} \exp{\left(-2\pi i \frac{\tanh{(tR|\lambda''|)}}{R \, \coth{(2tR|\lambda''|)}} \lambda'\cdot \left(Rx'+ \xi'\right)\right)} \, d\lambda \, dt, 
\end{align}
and then upon taking the limit as $R \to 0$, we get 
\begin{align*}
& \lim_{R\to 0} (\Lambda_R U)\left(R^2G_{\frac{e_1}{R}}+cI\right)^{-\alpha/2} (U^{-1}\Lambda_R^{-1}f)(x', x'') \\ 
& = \frac{(2\pi)^{-d_2}}{\Gamma(\alpha/2)} \int_{0}^\infty t^{\frac{\alpha}{2}-1} e^{-ct} \int_{\R^{d_1+d_2}} e^{-i\lambda'' \cdot x''} e^{-t} e^{-x_1'} e^{2\pi i x'\cdot \lambda'} e^{-|\lambda''|^2t} e^{-4\pi^2t|\lambda'|^2} \widehat{F}\left(\lambda',\frac{-\lambda''}{2\pi}\right) 
\, d\lambda \, dt \\
& = \frac{1}{\Gamma(\alpha/2)} \int_{0}^\infty t^{\frac{\alpha}{2}-1} e^{-ct} \int_{\R^{d_1+d_2}} e^{-t} e^{-x_1'} e^{2\pi i x \cdot \lambda} e^{-4\pi^2t|\lambda|^2} \widehat{F}\left(\lambda\right) \, d\lambda \, dt \\
& = (\Delta_{e_1}+cI)^{-\alpha/2}f(x),
\end{align*}
where we have used the generalized Lebesgue dominated convergence theorem to take the limit inside the integral, and this completes the proof of \eqref{eq:limit_giving_neg_result} and hence of part $(iii)$.
\end{proof}


\section*{Acknowledgments}
We are thankful to Sayan Bagchi for a number of insightful discussions throughout the development of this work. First author is grateful to Indian Institute of Science Education and Research (IISER) Bhopal for the Senior Research Fellowship (SRF). Second author was partially supported by the Anusandhan National Research Foundation (ANRF), India, under the research project ANRF/ARG/2025/003732/MS.


\providecommand{\bysame}{\leavevmode\hbox to3em{\hrulefill}\thinspace}
\providecommand{\MR}{\relax\ifhmode\unskip\space\fi MR }
\providecommand{\MRhref}[2]{%
  \href{http://www.ams.org/mathscinet-getitem?mr=#1}{#2}
}
\providecommand{\href}[2]{#2}

\end{document}